\documentclass[11pt]{article}
\usepackage{amsmath,amsmath,amsthm,amssymb}
\usepackage{graphicx}
\usepackage{enumerate}

\numberwithin{equation}{section}

\newtheorem {proposition}{Proposition}
\newtheorem{theorem}{Theorem}

\newtheorem {corollary}{Corollary}
\newtheorem {lemma}{Lemma}
\newtheorem {remark}{Remark}
\newtheorem*{definition}{Definition}

\newcommand{\ds}{\displaystyle}

\begin{document}
\title{Qualitative properties of solutions for an integral system related to the Hardy--Sobolev inequality}

\author{John Villavert\footnote{email: john.villavert@gmail.com, villavert@math.ou.edu} \\
[0.2cm] {\small Department of Mathematics, University of Oklahoma}\\
{\small Norman, OK 73019, USA}
 }
\date{}
\maketitle
\begin{abstract}
This article carries out a qualitative analysis on a system of integral equations of the Hardy--Sobolev type. Namely, results concerning Liouville type properties and the fast and slow decay rates of positive solutions for the system are established. For a bounded and decaying positive solution, it is shown that it either decays with the slow rates or the fast rates depending on its integrability. Particularly, a criterion for distinguishing integrable solutions from other bounded and decaying solutions in terms of their asymptotic behavior is provided. Moreover, related results on the optimal integrability, boundedness, radial symmetry and monotonicity of positive integrable solutions are also established. As a result of the equivalence between the integral system and a system of poly-harmonic equations under appropriate conditions, the results translate over to the corresponding poly-harmonic system. Hence, several classical results on semilinear elliptic systems are recovered and further generalized.
\end{abstract}

 \medskip

 
\maketitle
\section{Introduction and the main results}
In this paper, we study the qualitative properties of positive solutions for an integral system of the Hardy--Sobolev type. In particular, we consider the system of integral equations involving Riesz potentials and Hardy terms,
\begin{equation}\label{whls ie}
  \left\{\begin{array}{cl}
    u(x) = \ds\int_{\mathbb{R}^{n}} \ds\frac{v(y)^q}{|x-y|^{n-\alpha}|y|^{\sigma_1}}\,dy, \medskip \\ 
    v(x) = \ds\int_{\mathbb{R}^{n}} \ds\frac{u(y)^p}{|x-y|^{n-\alpha}|y|^{\sigma_2}}\,dy, \\
  \end{array}
\right.
\end{equation}
where throughout we assume that $n\geq 3$, $p,q > 0$ with $pq > 1$, $\alpha \in (0,n)$ and $\sigma_1,\sigma_2 \in [0,\alpha)$. As a result, we establish analogous properties for the closely related system of semilinear differential equations with singular weights,
\begin{equation}\label{whls pde}
  \left\{\begin{array}{cl}
    (-\Delta)^{\alpha/2} u(x) = \ds\frac{v(x)^q}{|x|^{\sigma_1}} & 	\text{ in } \mathbb{R}^n\backslash\{0\}, \medskip \\ 
    (-\Delta)^{\alpha/2} v(x) = \ds\frac{u(x)^p}{|x|^{\sigma_2}} & 	\text{ in } \mathbb{R}^n\backslash\{0\},
  \end{array}
\right.
\end{equation}
since both systems are equivalent under appropriate conditions. Particularly, if $p,q > 1$ and $\alpha = 2k$ is an even positive integer, then a positive classical solution $u,v \in C^{2k}(\mathbb{R}^n\backslash\{0\})\cap C(\mathbb{R}^n)$ of system \eqref{whls ie}, multiplied by suitable constants if necessary, satisfies the poly-harmonic system \eqref{whls pde} pointwise except at the origin; and vice versa (cf. \cite{CL13,Villavert:14c,WeiXu99}).

Our aim is to fully characterize the positive solutions, specifically the ground states, in terms of their asymptotic behavior and elucidate its connection with Liouville type non-existence results. The motivation for studying these properties for the Hardy--Sobolev type systems arises from several related and well-known problems. For example, one problem originates from the doubly weighted Hardy--Littlewood--Sobolev (HLS) inequality \cite{SW58}, which states that for $r,s \in (1,\infty)$, $\alpha \in (0,n)$ and $0 \leq \sigma_1 + \sigma_2 \leq \alpha$,
$$ \Big| \int_{\mathbb{R}^n}\int_{\mathbb{R}^n} \frac{f(x)g(y)}{|x|^{\sigma_1}|x-y|^{n-\alpha}|y|^{\sigma_2}} \,dxdy \Big| \leq C_{\sigma_{i},s,\alpha,n}\|f\|_{L^{r}(\mathbb{R}^n)}\|g\|_{L^{s}(\mathbb{R}^n)},$$
where 
$$ \frac{\alpha}{n} - \frac{1}{r} < \frac{\sigma_1}{n} < 1 - \frac{1}{r} \,\text{ and }\, \frac{1}{r} + \frac{1}{s} + \frac{\sigma_1 + \sigma_2}{n} = \frac{n+\alpha}{n}. $$
Here and throughout this paper, $\|f\|_{L^{p}(\mathbb{R}^n)}$ or $\|f\|_{p}$ denotes the norm of $f$ in the Lebesgue space $L^{p}(\mathbb{R}^n)$. To find the best constant in the doubly weighted HLS inequality, we maximize the functional
\begin{equation}\label{J functional}
J(f,g) = \int_{\mathbb{R}^n}\int_{\mathbb{R}^n} \frac{f(x)g(y)}{|x|^{\sigma_1}|x-y|^{n-\alpha}|y|^{\sigma_2}} \,dxdy
\end{equation}
under non-negative functions $f$ and $g$ with the constraints 
$$\|f\|_{L^{r}(\mathbb{R}^n)} = \|g\|_{L^{s}(\mathbb{R}^n)} = 1. $$ 
Setting $u= c_{1}f^{r-1}$ and $v = c_{2}g^{s-1}$ with proper choices for the constants $c_1$ and $c_2$ and taking $\frac{1}{p+1} = 1 - \frac{1}{r}$ and $\frac{1}{q+1} = 1 - \frac{1}{s}$ with $pq \neq 1$, the corresponding Euler--Lagrange equations for the extremal functions of the functional are equivalent to the so-called weighted HLS integral system
\begin{equation}\label{whls el}
  \left\{\begin{array}{cl}
    u(x) = \ds\frac{1}{|x|^{\sigma_1}}\int_{\mathbb{R}^{n}} \frac{v(y)^q}{|x-y|^{n-\alpha}|y|^{\sigma_2}}\,dy, \medskip \\ 
    v(x) = \ds\frac{1}{|x|^{\sigma_2}} \int_{\mathbb{R}^{n}} \frac{u(y)^p}{|x-y|^{n-\alpha}|y|^{\sigma_1}}\,dy, \\
  \end{array}
\right.
\end{equation}
where $$ \frac{\sigma_1}{n} < \frac{1}{p+1} < \frac{n-\alpha + \sigma_1}{n} \,\text{ and }\, \frac{1}{1+q}+\frac{1}{1+p} = \frac{n-\alpha+\sigma_1 + \sigma_2}{n}.$$
Notice that \eqref{whls ie} and \eqref{whls el} coincide if $\sigma_1 = \sigma_2 = 0$. Now when $\sigma_1 = \sigma_2 = 0$ and $r=s=\frac{2n}{n+\alpha}$, Lieb classified all maximizers of the functional \eqref{J functional} and posed the classification of all the critical points as an open problem in \cite{Lieb83}, which was later settled by Chen, Li and Ou in \cite{CLO06}.

If we take $\sigma_1 = \sigma_2 \doteq \sigma$, $p=q$ and $u\equiv v$, system \eqref{whls ie} becomes the integral equation
\begin{equation}\label{integral equation}
u(x) = \int_{\mathbb{R}^n} \frac{u(y)^p}{|x-y|^{n-\alpha}|y|^{\sigma}} \,dy.
\end{equation}
In the special case where $\alpha = 2$ and $p=\frac{n+\alpha-2\sigma}{n-\alpha}$, \eqref{integral equation} is closely related to the Euler--Lagrange equation for the extremal functions of the classical Hardy--Sobolev inequality, which states there exists a constant $C$ for which
\begin{equation*}
\Big( \int_{\mathbb{R}^n} \frac{u(x)^{\frac{2(n-\sigma)}{n-2}}}{|x|^{\sigma}} \Big)^{\frac{n-2}{n-\sigma}} \leq C\int_{\mathbb{R}^n} |\nabla u(x)|^2 \,dx \,\text{ for all }\, u \in \mathcal{D}^{1,2}(\mathbb{R}^n).
\end{equation*}
In fact, the Hardy--Sobolev inequality is a special case of the Caffarelli--Kohn--Nirenberg inequality (cf. \cite{BadialeTarantello02,CKN84,CW01,ChouChu93}). Furthermore, the classification of solutions for the unweighted integral equation and its corresponding differential equation provide an important ingredient in the Yamabe and prescribing scalar curvature problems. 

Another noteworthy and related issue concerns Liouville type theorems. Such non-existence results are important in deriving singularity estimates and a priori bounds for solutions of Dirichlet problems for a class of elliptic equations (cf. \cite{GS81apriori,PQS07}). The H\'{e}non--Lane--Emden system, which coincides with \eqref{whls pde} when $\alpha = 2$ and $\sigma_1,\sigma_2 \in (-\infty,2)$, has garnered some recent attention with respect to the H\'{enon}--Lane--Emden conjecture, which states that the system admits no positive classical solution in the subcritical case
$$ \frac{n-\sigma_1}{1+q} + \frac{n-\sigma_2}{1+p} > n-2.$$
In \cite{Phan12}, Phan proved the conjecture for $n=3$ in the class of bounded solutions and for $n=3,4$ provided that $\sigma_1,\sigma_2 \geq 0$ (see also \cite{FG14}). For the higher dimensional case or the general poly-harmonic version, this conjecture has partial results (cf. \cite{Fazly14,Phan12,Villavert:14c} and the references therein). In \cite{Villavert:14b}, the author verified that the poly-harmonic version is indeed sharp by establishing the existence of positive solutions for \eqref{whls pde} in the non-subcritical case (see also \cite{Li13,LGZ06a}). Even in the unweighted case (i.e. $\sigma_1 = \sigma_2 = 0$) with $\alpha = 2$, the conjecture, more commonly known as the Lane--Emden conjecture, still has only partial results. Specifically, it is true for radial solutions \cite{Mitidieri96} and for $n\leq 4$ \cite{PQS07,SZ96,Souplet09} (cf. \cite{AYZ14,LGZ06} for the poly-harmonic case).

In \cite{YLi92} and \cite{LiNi88}, the authors examined the decay properties of positive solutions for the Lane--Emden equation,
\begin{equation*}
-\Delta u(x) = u(x)^p \,\text{ in }\, \mathbb{R}^n.
\end{equation*}
It was shown that solutions decay to zero at infinity with either the fast rate $u(x) \simeq |x|^{-(n-2)}$ or with the slow rate $u(x) \simeq |x|^{-\frac{2}{p-1}}$, where the notation $f(x) \simeq g(x)$ means there exist positive constants $c$ and $C$ such that $cg(x) \leq f(x) \leq Cg(x)$ as $|x|\longrightarrow \infty$. Analogous studies on the asymptotic properties of solutions for the weighted integral equation \eqref{integral equation} can be found in \cite{Lei13} and for the unweighted version of system \eqref{whls ie} in \cite{LL13a}. In addition, results on the regularity of solutions for these equations and systems can be found in various papers (cf. \cite{CL05,CLO05a,Hang07,YYLi04,LuZhu11}). 

We are now ready to describe the main results of this paper. Henceforth, we define
\begin{equation*}
p_0 \doteq \frac{\alpha(1+p) - (\sigma_2 + \sigma_1 p)}{pq-1} \,\text{ and }\, q_0 \doteq \frac{\alpha(1+q) - (\sigma_1 + \sigma_2 q)}{pq-1}.
\end{equation*}
We also define some notions of solutions for the Hardy--Sobolev type system including the integrable solutions. 
\begin{definition}
Let $u,v$ be positive solutions of \eqref{whls ie}. Then $u,v$ are said to be:
\begin{enumerate}[(i)]
\item \textbf{decaying} solutions if $u(x) \simeq |x|^{-\theta_1}$ and $v(x) \simeq |x|^{-\theta_2}$ for some rates $\theta_{1},\theta_{2} > 0$;
\item \textbf{integrable} solutions if $u \in L^{r_0}(\mathbb{R}^n)$ and $v \in L^{s_0}(\mathbb{R}^n)$
where
\begin{equation*}
r_0 = \frac{n}{q_0} ~\,\text{ and }~\, s_0 = \frac{n}{p_0}.
\end{equation*}
\end{enumerate}
\end{definition}

\begin{definition}
Let $u,v$ be positive solutions of \eqref{whls ie}. Then $u,v$ are said to decay with the \textbf{slow rates} as $|x| \longrightarrow \infty$ if $u(x) \simeq |x|^{-q_0}$ and $v(x) \simeq |x|^{-p_0}$.
Suppose $q\geq p$ and $\sigma_1 \geq \sigma_2$. Then $u,v$ are said to decay with the \textbf{fast rates} as $|x| \longrightarrow \infty$ if 
$$ u(x) \simeq |x|^{-(n-\alpha)} $$
and
\begin{equation*}
\left\{\begin{array}{ll}
v(x) \simeq |x|^{-(n-\alpha)}, 							& \text{ if }\, p(n-\alpha) + \sigma_2 > n; \\
v(x) \simeq |x|^{-(n-\alpha)}\ln |x|, 					& \text{ if }\, p(n-\alpha) + \sigma_2 = n; \\
v(x) \simeq |x|^{-(p(n-\alpha) - (\alpha-\sigma_2))}, 	& \text{ if }\, p(n-\alpha) + \sigma_2 < n.
\end{array}
\right.
\end{equation*}
\end{definition}

\begin{remark}
The conditions on the parameters in the previous definition, i.e. $q \geq p$ and $\sigma_1 \geq \sigma_2$, which we will sometimes assume within our main theorems, are not so essential. Namely, the results still remain true if we interchange the parameters provided that $u$ and $v$ are interchanged accordingly in the definition and the theorems.
\end{remark}

\begin{remark}
In view of the equivalence with poly-harmonic systems and the regularity theory indicated by the earlier references, \textbf{classical solutions} of \eqref{whls ie} should be understood to mean solutions belonging to $C^{\lfloor \alpha \rfloor}(\mathbb{R}^n \backslash \{0\})\cap C(\mathbb{R}^n)$, where $\lfloor\, \cdot \,\rfloor$ is the greatest integer function.
\end{remark}

\subsection{Main results}

\begin{theorem}\label{slow decay} There hold the following.

\begin{enumerate}[(i)]
\item If $u,v$ are bounded and decaying solutions of \eqref{whls ie}, then there exists a positive constant $C$ such that as $|x|\longrightarrow \infty$,
$$u(x) \leq C|x|^{-q_0} \,\text{ and }\, v(x) \leq C|x|^{-p_0}.$$

\item Suppose $q\geq p$ and $\sigma_1\geq \sigma_2$ (so that $q_0 \geq p_0$), and let $u,v$ be positive solutions of \eqref{whls ie}. Then there exists a positive constant $c$ such that as $|x|\longrightarrow \infty$,
$$u(x) \geq \frac{c}{|x|^{n-\alpha}} \,\text{ and }\, v(x) \geq \frac{c}{|x|^{\min\{ n-\alpha, p(n-\alpha) - (\alpha - \sigma_2)\}}}.$$
\end{enumerate}
\end{theorem}

We now introduce our Liouville type theorem for the Hardy--Sobolev type system. Basically, our result states that the system has no non-negative ground state classical solutions in the subcritical case besides the trivial pair $u,v \equiv 0$.
\begin{theorem}\label{Liouville}
Suppose that, in addition, $\alpha \in (1,n)$. Then system \eqref{whls ie} does not admit any bounded and decaying positive classical solution whenever the following subcritical condition holds
\begin{equation}\label{subcritical}
\frac{n-\sigma_1}{1+q} + \frac{n-\sigma_2}{1+p} > n-\alpha.
\end{equation}
\end{theorem}
\begin{remark}
In \cite{Villavert:14c}, the author proved the following non-existence theorem but without any boundedness or decaying assumption on positive solutions.\medskip

\noindent{\bf Theorem A.}
\it System \eqref{whls ie} has no positive solution if either $pq \in (0,1]$ or when $pq > 1$ and $\max \lbrace p_0, q_0 \rbrace \geq n-\alpha$.\medskip

Interestingly, Theorem A and Theorem \ref{Liouville} are reminiscent of the non-existence results of Serrin and Zou \cite{SZ96} for the Lane--Emden system. We should also mention the earlier work in \cite{DAmbrosioMitidieri14}, which also obtained similar Liouville theorems among other interesting and related results. 
\end{remark}

Theorem \ref{Liouville} applies to integrable solutions as well. In particular, we later show that integrable solutions are indeed radially symmetric and decreasing about the origin. Therefore, the proof of Theorem \ref{Liouville} can be adopted in this situation to get the following. 
\begin{corollary}\label{Liouville integrable}
Suppose that, in addition, $\alpha \in (1,n)$. Then system \eqref{whls ie} does not admit any positive integrable solution whenever the subcritical condition \eqref{subcritical} holds.
\end{corollary}

The next theorem concerns the properties of integrable solutions. More precisely, it states that we can distinguish integrable solutions from the other ground states with ``slower" decay rates in that integrable solutions are equivalently characterized as the bounded positive solutions which decay with the fast rates. In dealing with integrable solutions, Corollary \ref{Liouville integrable} indicates that we can restrict our attention to the non-subcritical case:
\begin{equation}\label{non-subcritical0}
q_0 + p_0 \leq n-\alpha,
\end{equation}
or equivalently
\begin{equation}\label{non-subcritical}
\frac{n-\sigma_1}{1+q} + \frac{n-\sigma_2}{1+p} \leq n-\alpha.
\end{equation}
We shall also assume that $p,q > 1$. Moreover, the theorem further asserts that there are also no positive integrable solutions in the supercritical case.

\begin{theorem}\label{integrable theorem}
Suppose $q \geq p > 1$, $\sigma_1 \geq \sigma_2$ (so that $q_0 \geq p_0$), and let $u,v$ be positive solutions of \eqref{whls ie} satisfying the non-subcritical condition \eqref{non-subcritical}. 

\noindent (i) Then $u,v$ are integrable solutions if and only if $u,v$ are bounded and decay with the \textbf{fast rates} as $|x| \longrightarrow \infty$:

\begin{equation*}
\left\{\begin{array}{ll}
u(x) \simeq |x|^{-(n-\alpha)}; \\
v(x) \simeq |x|^{-(n-\alpha)}, 							& \text{ if }\, p(n-\alpha) + \sigma_2 > n; \\
v(x) \simeq |x|^{-(n-\alpha)}\ln |x|, 					& \text{ if }\, p(n-\alpha) + \sigma_2 = n; \\
v(x) \simeq |x|^{-(p(n-\alpha) - (\alpha-\sigma_2))}, 	& \text{ if }\, p(n-\alpha) + \sigma_2 < n.
\end{array}
\right.
\end{equation*}
(ii) Suppose that, in addition, $\alpha \in (1,n)$. If $u,v$ are integrable solutions of \eqref{whls ie}, then the critical condition,
\begin{equation}\label{supercritical}
\frac{n-\sigma_1}{1+q} + \frac{n-\sigma_2}{1 + p} = n-\alpha,
\end{equation}
necessarily holds.
\end{theorem}

\begin{remark}
(i) Notice that if $\sigma_1 = \sigma_2 = 0$, the Hardy--Sobolev type system \eqref{whls ie} coincides with the HLS system \eqref{whls el}. Indeed, our main results on the integrable solutions, including the subsequent results below on the optimal integrability, boundedness, radial symmetry and monotonicity of solutions, coincide with and thus extend past results of \cite{CL09,LL13a} (also see \cite{Caristi2008,LL13,LLM12,ZhaoLei2012} for closely related results).

(ii) To further illustrate their connection, we remark that the integrable solutions in the unweighted case turn out to be the finite-energy solutions of the HLS system \eqref{whls el}, i.e. the critical points $(u,v) \in L^{p+1}(\mathbb{R}^n) \times L^{q+1}(\mathbb{R}^n)$ for the HLS functional.

(iii) In the weighted case, however, system \eqref{whls ie} cannot be recovered from \eqref{whls el} due to their different singular weights. In fact, the asymptotic behavior of positive solutions between the two are \underline{not} the same (cf. \cite{LeiMa11,LiLim07}).
\end{remark}

Our last main result asserts that if $u,v$ are bounded and positive but are not integrable solutions, then they ``decay with almost the slow rates" and we conjecture that they actually do converge with the slow rates. Of course, if this conjecture were true, then Theorem \ref{Liouville} would hold without any additional decaying assumption on solutions. However, if, in addition, $u,v$ are decaying solutions, then they do indeed decay with the slow rates. 

\begin{theorem}\label{slow theorem}
Let $u,v$ be bounded positive solutions of \eqref{whls ie}. Then there hold the following. 
\begin{enumerate}[(i)]
\item There does not exist a positive constant $c$ such that
$$\text{either }\, u(x) \geq c(1+|x|)^{-\theta_1} \,\text{ or }\, v(x)\geq c(1+|x|)^{-\theta_2},$$
where $\theta_1 < q_0,~ \theta_2 < p_0$. 

\item If $u,v$ are not integrable solutions, then $u,v$ decay with rates not faster than the slow rates. Namely, there does not exist a positive constant $C$ such that 
$$\text{either }\, u(x) \leq C(1+|x|)^{-\theta_3} \,\text{ or }\, v(x) \leq C(1+|x|)^{-\theta_4},$$
where $\theta_3 > q_0,~ \theta_4 > p_0$. 

\item If $u,v$  are not integrable solutions but are decaying solutions, then $u,v$ must necessarily have the slow rates as $|x| \longrightarrow \infty$.
\end{enumerate}
\end{theorem}

\begin{remark}
For the sake of conciseness, rather than formally stating the corresponding results for the poly-harmonic system \eqref{whls pde} as corollaries, we only point out that our main results for the integral system do translate over to system \eqref{whls pde} provided that the equivalence conditions described earlier are satisfied. 
\end{remark}

The remaining parts of this paper are organized in the following way. In section \ref{slow section}, some preliminary results are established in which Theorem \ref{slow decay} is an immediate consequence of. Then, we apply an integral form of a Pohozaev type identity to prove Theorem \ref{Liouville}. Section \ref{integrable section} establishes several key properties of integrable solutions: the optimal integrability, boundedness and convergence properties of integrable solutions, and we show the integrable solutions are radially symmetric and decreasing about the origin. Using these properties, we prove Theorem \ref{integrable theorem} in section \ref{fast section}. The paper then concludes with section \ref{non-integrable section}, which contains the proof of Theorem \ref{slow theorem}.

\section{Slow decay rates and the non-existence theorem}\label{slow section}
\begin{proposition}\label{slow decay 1} Let $u,v$ be bounded and decaying positive solutions of \eqref{whls ie}.
\begin{enumerate}[(i)]
\item There exists a positive constant $C$ such that as $|x| \longrightarrow \infty$
\begin{equation*}
u(x) \leq C|x|^{-q_0} \,\text{ and }\, v(x) \leq C|x|^{-p_0}.
\end{equation*}
\item Moreover, the improper integrals,
$$ \int_{\mathbb{R}^n} \frac{u(x)^{p+1}}{|x|^{\sigma_2}}\,dx \,\text{ and }\, \int_{\mathbb{R}^n} \frac{v(x)^{q+1}}{|x|^{\sigma_1}}\,dx,$$
are finite provided that the subcritical condition \eqref{subcritical} holds.
\end{enumerate}
\end{proposition}

\begin{proof}
For $|x| > 2R$ with $R>0$ suitably large,
\begin{align}\label{lu}
u(x) \geq {} & \int_{B_{|x|}(0)\backslash B_{R}(0)} \frac{v(y)^q}{|x-y|^{n-\alpha}|y|^{\sigma_1}}\,dy \notag \\
\geq {} & Cv(x)^{q}|x|^{\alpha - \sigma_1 -n} \int_{B_{|x|}(0)\backslash B_{R}(0)} \,dt \geq Cv(x)^{q}|x|^{\alpha-\sigma_1}.
\end{align}
Similarly, we can show 
\begin{equation}\label{lv}
v(x) \geq Cu(x)^{p}|x|^{\alpha-\sigma_2},
\end{equation}
and combining \eqref{lu} with \eqref{lv} gives us
\begin{align*}
u(x) \geq Cv(x)^q |x|^{\alpha-\sigma_1} \geq Cu(x)^{pq}|x|^{q(\alpha-\sigma_2)+(\alpha-\sigma_1)}, \\
v(x) \geq Cu(x)^p |x|^{\alpha-\sigma_2} \geq Cv(x)^{pq}|x|^{p(\alpha-\sigma_1)+(\alpha-\sigma_2)}.
\end{align*}
Indeed, these estimates imply $$u(x) \leq C|x|^{-q_0} \,\text{ and }\, v(x) \leq C|x|^{-p_0} \,\text{ as }\, |x| \longrightarrow \infty.$$
In addition,
\begin{align*}
\int_{\mathbb{R}^n} \frac{u(x)^{p+1}}{|x|^{\sigma_2}}\,dx \leq {} & \int_{B_{R}(0)} \frac{u(x)^{p+1}}{|x|^{\sigma_2}}\,dx + \int_{B_{R}(0)^C} \frac{u(x)^{p+1}}{|x|^{\sigma_2}}\,dx \\
\leq {} & C_{1} + C_{2}\int_{R}^{\infty} t^{ -q_0 (p+1) + n -\sigma_2} \,\frac{dt}{t} < \infty, 
\end{align*}
since \eqref{subcritical} implies $-q_0 (p+1) + n - \sigma_2 = n-\alpha - (q_0 + p_0) < 0$. Likewise, $\int_{\mathbb{R}^n} \frac{v(x)^{q+1}}{|x|^{\sigma_{1}}}\,dx < \infty$ using similar calculations, and this completes the proof.
\end{proof}

\begin{proposition}\label{slow decay 2}
Let $u,v$ be bounded positive solutions of \eqref{whls ie}. Then there exists a positive constant $C>0$ such that as $|x|\longrightarrow \infty$,
$$\ds u(x) \geq \frac{C}{(1+|x|)^{n-\alpha}} \,\,\text{ and }\,\, \ds v(x) \geq \frac{C}{(1+|x|)^{\min\lbrace n-\alpha,\, p(n-\alpha)-(\alpha -\sigma_2) \rbrace }}. $$
\end{proposition}

\begin{proof}
For $y \in B_{1}(0)$, we can find a $C>0$ such that 
$$ C \leq \int_{B_{1}(0)} \frac{v(y)^q}{|y|^{\sigma_1}} \,dy, \int_{B_{1}(0)}  \frac{u(y)^p}{|y|^{\sigma_2}} \,dy < \infty. $$
Thus for $x\in B_{1}(0)^C$, we have
\begin{align}\label{lower u}
u(x) \geq {} & \int_{B_{1}(0)} \frac{v(y)^q}{|x-y|^{n-\alpha}|y|^{\sigma_1}} \,dy \notag \\
\geq {} & \frac{C}{(1+|x|)^{n-\alpha}}\int_{B_{1}(0)} \frac{v(y)^q}{|y|^{\sigma_1}}\,dy \geq \frac{C}{(1+|x|)^{n-\alpha}}.
\end{align}
Similarly, we can show 
\begin{equation*}
v(x) \geq \frac{C}{(1+|x|)^{n-\alpha}}.
\end{equation*}
Then, with the help of estimate \eqref{lower u}, we get
\begin{align*}
v(x) \geq {} & \int_{B_{|x|/2}(x)} \frac{u(y)^p}{|x-y|^{n-\alpha}|y|^{\sigma_2}}\,dy \\
\geq {} & \frac{C}{(1+|x|)^{p(n-\alpha)+\sigma_2}}\int_{0}^{|x|/2} t^{\alpha} \,\frac{dt}{t} = \frac{C}{(1+|x|)^{p(n-\alpha)-(\alpha-\sigma_2)}}.
\end{align*}
\end{proof}

\begin{proof}[Proof of Theorem \ref{slow decay}]
This is a direct consequence of Proposition \ref{slow decay 1}(i) and Proposition \ref{slow decay 2}.
\end{proof}

\begin{proof}[Proof of Theorem \ref{Liouville}]
We proceed by contradiction. That is, assume $u,v$ are bounded and decaying positive classical solutions. First, notice that integration by parts implies
\begin{align*}
\int_{B_{R}(0)} {} & \frac{v(x)^q}{|x|^{\sigma_1}} (x\cdot \nabla v(x)) + \frac{u(x)^p}{|x|^{\sigma_2}} (x\cdot \nabla u(x)) \,dx \notag \\
= {} & \frac{1}{1+q}\int_{B_{R}(0)} \frac{x}{|x|^{\sigma_1}}\cdot\nabla (v(x)^{q+1})\,dx  + \frac{1}{1+p}\int_{B_{R}(0)} \frac{x}{|x|^{\sigma_2}}\cdot\nabla (u(x)^{p+1})\,dx \notag \\
= {} & -\frac{n-\sigma_1}{1+q}\int_{B_{R}(0)} \frac{v(x)^{q+1}}{|x|^{\sigma_1}}\,dx - \frac{n-\sigma_2}{1+p}\int_{B_{R}(0)} \frac{u(x)^{p+1}}{|x|^{\sigma_2}}\,dx \notag \\
+ {} & \frac{R}{1+q}\int_{\partial B_{R}(0)} \frac{v(x)^{q+1}}{|x|^{\sigma_1}}\,ds + \frac{R}{1+p}\int_{\partial B_{R}(0)} \frac{u(x)^{p+1}}{|x|^{\sigma_2}}\,ds.
\end{align*}
Note that this identity follows more precisely by integrating on $B_{R}(0)\backslash B_{\varepsilon}(0)$ then sending $\varepsilon \longrightarrow 0$ after the appropriate calculations. Then, by virtue of Proposition \ref{slow decay 1}(ii), we can find a sequence $\lbrace R_j \rbrace$ such that as $R_j \longrightarrow \infty$,
\begin{equation*}
R_{j}\int_{\partial B_{R_j}(0)} \frac{v(x)^{q+1}}{|x|^{\sigma_1}}\,ds, \,R_{j}\int_{\partial B_{R_j}(0)} \frac{u(x)^{p+1}}{|x|^{\sigma_2}}\,ds \longrightarrow 0.
\end{equation*}
Thus, we obtain the identity
\begin{align}\label{intparts sys}
\int_{\mathbb{R}^n} \frac{v(x)^q}{|x|^{\sigma_1}} (x\cdot \nabla v(x)) + {} & \frac{u(x)^p}{|x|^{\sigma_2}} (x\cdot \nabla u(x)) \,dx \notag \\
= {} & -\Bigg\lbrace \frac{n-\sigma_1}{1+q} + \frac{n-\sigma_2}{1+p}\Bigg\rbrace\int_{\mathbb{R}^n} \frac{v(x)^{q+1}}{|x|^{\sigma_1}}\,dx,
\end{align}
where we used the fact that
\begin{equation*}
\int_{\mathbb{R}^n} \frac{u(x)^{p+1}}{|x|^{\sigma_2}}\,dx = \int_{\mathbb{R}^n}\int_{\mathbb{R}^n} \frac{u(x)^{p}v(z)^q}{|x-z|^{n-\alpha}|x|^{\sigma_2}|z|^{\sigma_1}}\,dz dx = \int_{\mathbb{R}^n} \frac{v(x)^{q+1}}{|x|^{\sigma_1}}\,dx.
\end{equation*}
From the first equation with $\lambda \neq 0$, we write
\begin{equation*}
u(\lambda x) = \int_{\mathbb{R}^n} \frac{v(y)^{q}}{|\lambda x - y|^{n-\alpha}|y|^{\sigma_1}}\,dy = \lambda^{\alpha - \sigma_1}\int_{\mathbb{R}^n} \frac{v(\lambda z)^q}{|x-z|^{n-\alpha}|z|^{\sigma_1}}\,dz.
\end{equation*}
Differentiating this rescaled equation with respect to $\lambda$ then taking $\lambda = 1$ gives us
\begin{align}\label{xdotwithu}
x\cdot \nabla u(x) = {} & (\alpha - \sigma_1)\int_{\mathbb{R}^n}\frac{v(z)^q}{|x-z|^{n-\alpha}|z|^{\sigma_1}}\,dz + \int_{\mathbb{R}^n} \frac{q v(z)^{q-1}(z\cdot \nabla v)}{|x-z|^{n-\alpha}|z|^{\sigma_1}}\,dz \notag \\
= {} & (\alpha - \sigma_1)u(x) + \int_{\mathbb{R}^n} \frac{z\cdot \nabla v(z)^{q}}{|x-z|^{n-\alpha}|z|^{\sigma_1}}\,dz ~ (x\neq 0). 
\end{align}
Note that an integration by parts yields
\begin{align*}
\int_{B_{R}(0)} {} & \frac{z\cdot \nabla v(z)^{q}}{|x-z|^{n-\alpha}|z|^{\sigma_1}}\,dz \\
= {} & R\int_{\partial B_{R}(0)} \frac{v(z)^{q}}{|x-z|^{n-\alpha}|z|^{\sigma_1}} \,ds 
- (n-\sigma_1)\int_{B_{R}(0)} \frac{v(z)^{q}}{|x-z|^{n-\alpha}|z|^{\sigma_1}}\,dz \\
{} & - (n-\alpha)\int_{B_{R}(0)} \frac{(z\cdot (x-z))v(z)^{q}}{|x-z|^{n-\alpha + 2}|z|^{\sigma_1}}\,dz.
\end{align*}
By virtue of $\int_{\mathbb{R}^n} \frac{v(y)^{q}}{|x-z|^{n-\alpha}|z|^{\sigma_1}} \,dz < \infty,$ we can find a sequence $\lbrace R_j \rbrace$ such that 
\begin{equation*}
R_{j}\int_{\partial B_{R_j}(0)} \frac{v(z)^{q}}{|x-z|^{n-\alpha}|z|^{\sigma_1}} \,ds  \longrightarrow 0 \,\text{ as }\, R_j \longrightarrow \infty
\end{equation*}
and thus obtain
\begin{align*}
\int_{\mathbb{R}^n} \frac{z\cdot \nabla v(z)^{q}}{|x-z|^{n-\alpha}|z|^{\sigma_1}}\,dz = {} & - (n-\sigma_1)\int_{\mathbb{R}^n} \frac{v(z)^{q}}{|x-z|^{n-\alpha}|z|^{\sigma_1}}\,dz \\
{} & - (n-\alpha)\int_{\mathbb{R}^n} \frac{(z\cdot (x-z))v(z)^{q}}{|x-z|^{n-\alpha + 2}|z|^{\sigma_1}}\,dz.
\end{align*}
Hence, inserting this into \eqref{xdotwithu} yields
\begin{equation}\label{intparts 1}
x\cdot \nabla u(x) = -(n-\alpha)u(x) - (n-\alpha)\int_{\mathbb{R}^n} \frac{(z\cdot (x-z))v(z)^{q}}{|x-z|^{n-\alpha + 2}|z|^{\sigma_1}}\,dz.
\end{equation}
Similar calculations on the second integral equation will lead to 
\begin{equation}\label{intparts 2}
x\cdot \nabla v(x) = -(n - \alpha)v(x) - (n-\alpha)\int_{\mathbb{R}^n} \frac{(z\cdot (x-z))u(z)^{p}}{|x-z|^{n-\alpha + 2}|z|^{\sigma_2}}\,dz.
\end{equation}
Now multiply \eqref{intparts 1} and \eqref{intparts 2} by $|x|^{-\sigma_2}u(x)^{p}$ and $|x|^{-\sigma_1}v(x)^{q}$, respectively, sum the resulting equations together and integrate over $\mathbb{R}^n$ to get
\begin{align*}
\int_{\mathbb{R}^n} \frac{v(x)^{q}}{|x|^{\sigma_1}} {} &  (x\cdot \nabla v(x))\,dx + \int_{\mathbb{R}^n} \frac{u(x)^{p}}{|x|^{\sigma_2}}(x\cdot \nabla u(x))\,dx \notag \\
= {} & -(n-\alpha)\Bigg\lbrace\int_{\mathbb{R}^n} \frac{v(x)^{q+1}}{|x|^{\sigma_1}} + \frac{u(x)^{p+1}}{|x|^{\sigma_2}}\,dx\Bigg\rbrace \notag \\
{} & - (n-\alpha)\int_{\mathbb{R}^n}\int_{\mathbb{R}^n} \frac{(z\cdot (x-z) + x\cdot(z-x))u(z)^{p}v(x)^{q}}{|x-z|^{n-\alpha + 2}|z|^{\sigma_2}|x|^{\sigma_1}}\,dz dx.
\end{align*}
By noticing that $z\cdot(x-z) + x\cdot(z-x) = -|x-z|^2$, we obtain the Pohozaev type identity
\begin{equation}\label{pohozaev sys}
\int_{\mathbb{R}^n} \frac{v(x)^{q}}{|x|^{\sigma_1}}(x\cdot \nabla v(x)) + \frac{u(x)^{p}}{|x|^{\sigma_2}}(x\cdot \nabla u(x))\,dx 
= -(n-\alpha)\int_{\mathbb{R}^n} \frac{v(x)^{q+1}}{|x|^{\sigma_1}}\,dx.
\end{equation}
Inserting \eqref{intparts sys} into \eqref{pohozaev sys} yields
\begin{equation*}
\Bigg\lbrace \frac{n-\sigma_1}{1+q} + \frac{n-\sigma_2}{1+p} - (n-\alpha) \Bigg\rbrace \int_{\mathbb{R}^n} \frac{v(x)^{q+1}}{|x|^{\sigma_1}}\,dx 
= 0,
\end{equation*}
but this contradicts with \eqref{subcritical}. This completes the proof of the theorem.
\end{proof}

\section{Properties of integrable solutions}\label{integrable section}

\subsection{An equivalent form of the weighted HLS inequality}
The following estimate is a consequence of the doubly weighted HLS inequality by duality, and it is the version of the weighted HLS inequality we apply in this paper.
\begin{lemma}
Let $p,q \in (1,\infty)$, $\alpha \in (0,n)$ and $0\leq \sigma_1 + \sigma_2 \leq \alpha$, and define 
$$I_{\alpha} f(x) \doteq \int_{\mathbb{R}^n} \frac{f(y)}{|x|^{\sigma_1}|x-y|^{n-\alpha}|y|^{\sigma_2}}\,dy.$$ 
Then
$$\| I_{\alpha} f\|_{L^{q}(\mathbb{R}^n)} \leq C_{\sigma_{i},p,\alpha,n}\|f\|_{L^{p}(\mathbb{R}^n)},$$
where $\frac{1}{p}-\frac{1}{q} = \frac{\alpha - (\sigma_1 + \sigma_2)}{n}$  and  $\frac{1}{q} - \frac{n-\alpha}{n} < \frac{\sigma_1}{n} < \frac{1}{q}$.
\end{lemma}

\subsection{Integrability of solutions}

\begin{theorem}\label{integrability}
Suppose  $q \geq p > 1$ and $\sigma_1 \geq \sigma_2$. If $u$ and $v$ are positive integrable solutions of \eqref{whls ie}, then $(u,v) \in  L^{r}(\mathbb{R}^n) \times L^{s}(\mathbb{R}^n)$ for each pair $(r,s)$ such that
$$ \frac{1}{r} \in \Big(0,\frac{n-\alpha}{n}\Big) \,\text{ and }\, \frac{1}{s} \in \Big(0,\min\Big\{ \frac{n-\alpha}{n},\frac{p(n-\alpha) - (\alpha-\sigma_2)}{n} \Big\}\Big).$$
\end{theorem}
\begin{proof}
\noindent{\it Step 1:} Establish an initial interval of integrability.

Set $a = \frac{1}{r_0} = \frac{\alpha(1 + q) - (\sigma_1 + \sigma_2 q)}{n(pq-1)}$, $b = \frac{1}{s_0} = \frac{\alpha(1 + p) - (\sigma_2 + \sigma_1 p)}{n(pq-1)}$ and let $\frac{1}{r} \in (a - b, \frac{n-\alpha}{n})$ and $\frac{1}{s} \in (0,\frac{n-\alpha}{n} - a + b)$ such that
$$\frac{1}{r} - \frac{1}{s} = \frac{1}{r_0} - \frac{1}{s_0}.$$
Thus, we have
\begin{equation}\label{indices}
\frac{1}{r} + \frac{\alpha-\sigma_1}{n} = \frac{q-1}{s_0} + \frac{1}{s} \,\text{ and }\, \frac{1}{s} + \frac{\alpha-\sigma_2}{n} = \frac{p-1}{r_0} + \frac{1}{r}. 
\end{equation}
Let $A>0$ and define $u_A = u$ if $u > A$ or $|x| > A$, $u_A = 0$ if $u\leq A$ and $|x| \leq A$; we give $v_A$ the analogous definition. Consider the integral operator $T = (T_1,T_2)$ where
$$ T_{1}g(x) = \int_{\mathbb{R}^n} \frac{v_{A}(y)^{q-1}g(y)}{|x-y|^{n-\alpha}|y|^{\sigma_1}}\,dy \,\text{ and }\, T_{2}f(x) = \int_{\mathbb{R}^n} \frac{u_{A}(y)^{p-1}f(y)}{|x-y|^{n-\alpha}|y|^{\sigma_2}}\,dy$$
for $f \in L^{r}(\mathbb{R}^n)$ and $g\in L^{s}(\mathbb{R}^n)$. By virtue of \eqref{indices}, applying the weighted Hardy--Littlewood--Sobolev inequality followed by H\"{o}lder's inequality gives us
\begin{align*}
\|T_{1}g\|_{L^r} \leq C\|v_{A}^{q-1}g\|_{\frac{nr}{n+r(\alpha-\sigma_1)}} \leq C\|v_A\|_{s_0}^{q-1}\|g\|_{s}, \\
\|T_{2}f\|_{L^s} \leq C\|u_{A}^{p-1}f\|_{\frac{ns}{n+s(\alpha-\sigma_2)}} \leq C\|u_A\|_{r_0}^{p-1}\|f\|_{r}.
\end{align*}
We may choose $A$ sufficiently large so that 
$$C\|v_A\|_{s_0}^{q-1},~ C\|u_A\|_{r_0}^{p-1} \leq 1/2$$
and the operator $T(f,g) = (T_{1}g,T_{2}f)$, equipped with the norm 
$$\|(f_1,f_2)\|_{L^{r}(\mathbb{R}^n)\times L^{s}(\mathbb{R}^n)} = \|f_1\|_r + \|f_2\|_{s},$$ 
is a contraction map from $L^{r}(\mathbb{R}^n)\times L^{s}(\mathbb{R}^n)$ to itself
$$ \text{ for all }\, (\frac{1}{r},\frac{1}{s}) \in I \doteq (a-b,\frac{n-\alpha}{n})\times (0,\frac{n-\alpha}{n} - a + b).$$
Thus, $T$ is also a contraction map from $L^{r_0}(\mathbb{R}^n)\times L^{s_0}(\mathbb{R}^n)$ to itself since $(\frac{1}{r_0},\frac{1}{s_0}) \in I$. Now define
$$ F = \int_{\mathbb{R}^n} \frac{(v - v_A)^q(y)}{|x-y|^{n-\alpha}|y|^{\sigma_1}}\,dy \,\text{ and }\, G = \int_{\mathbb{R}^n} \frac{(u - u_A)^p(y)}{|x-y|^{n-\alpha}|y|^{\sigma_2}}\,dy. $$ 
Then the weighted HLS inequality implies that $(F,G) \in L^{r}(\mathbb{R}^n)\times L^{s}(\mathbb{R}^n)$ and since $(u,v)$ satisfies
$$(f,g) = T(f,g) + (F,G),$$
applying Lemma 2.1 from \cite{JL05} yields
\begin{equation}\label{small interval}
(u,v) \in L^{r}(\mathbb{R}^n)\times L^{s}(\mathbb{R}^n) \,\text{ for all }\, (\frac{1}{r},\frac{1}{s}) \in I.
\end{equation}

\noindent{\it Step 2:} Extend the interval of integrability $I$.

First, we claim that
\begin{equation}\label{inequality}
q\Big\lbrace \frac{n-\alpha}{n} - a + b\Big\rbrace - \frac{\alpha - \sigma_1}{n} > a-b.
\end{equation}
If this holds true, we can then apply the weighted HLS inequality to obtain
$$ \|u\|_{r} \leq C\|v^q\|_{\frac{nr}{n+r(\alpha-\sigma_1)}} \leq C\|v\|_{\frac{nrq}{n+r(\alpha-\sigma_1)}}^q. $$
Then, since $v\in L^{s}(\mathbb{R}^n)$ for all $\frac{1}{s} \in (0,\frac{n-\alpha}{n} - a + b)$, we obtain that $u \in L^{r}(\mathbb{R}^n)$ for $\frac{1}{r} \in (0, q\lbrace \frac{n-\alpha}{n} - a + b\rbrace - \frac{\alpha - \sigma_1}{n} )$, where we are using \eqref{inequality} for the last interval to make sense. Combining this with \eqref{small interval} yields
\begin{equation}\label{r}
u \in L^{r}(\mathbb{R}^n) \,\text{ for all }\, \frac{1}{r} \in \Big(0,\frac{n-\alpha}{n}\Big).
\end{equation}
Likewise, the weighted HLS inequality implies
$$\|v\|_{s} \leq C\|u\|_{\frac{nsp}{n+s(\alpha-\sigma_2)}}^p.$$
Thus, combining this with \eqref{r} yields
$$ v \in L^{s}(\mathbb{R}^n) \,\text{ for all }\, \frac{1}{s} \in \Big(0,\min\Big\lbrace \frac{n-\alpha}{n},\frac{p(n-\alpha) - (\alpha-\sigma_2)}{n} \Big\rbrace\Big). $$ 
It remains to verify the claim \eqref{inequality}. To do so, notice that \eqref{non-subcritical0} implies that
\begin{align*}
q\Big\lbrace \frac{n-\alpha}{n} {} & - a + b \Big\rbrace - \frac{ \alpha - \sigma_1}{n} > qb - \frac{ \alpha - \sigma_1}{n} = \frac{\alpha q - \sigma_2 q + \alpha - \sigma_1 }{n(pq-1)} \\
= {} & \frac{\alpha (q-p) + \alpha p - \sigma_2 q + \alpha - \sigma_1 }{n(pq-1)} >  \frac{\alpha (q-p) + \sigma_1 p - \sigma_2 q + \sigma_2 - \sigma_1 }{n(pq-1)} \\
= {} &  \frac{\alpha (q-p) - \sigma_1(1-p) - \sigma_2 (q-1) }{n(pq-1)} = a - b.
\end{align*}
This completes the proof.
\end{proof}

\begin{remark}
Actually, it is not too difficult to show that the interval of integrability of Theorem \ref{integrability} is optimal. That is, $\|u\|_{r} = \infty$ and $\|v\|_{s} = \infty$ at the endpoints
$$ r = \frac{n}{n-\alpha} \,\text{ and }\, s = \max\Big\{ \frac{n}{n-\alpha},\frac{n}{p(n-\alpha) - (\alpha-\sigma_2)} \Big\}. $$
\end{remark}

\subsection{Integrable solutions are ground states}

\begin{theorem}
If $u,v$ are positive integrable solutions of \eqref{whls ie}, then $u,v$ are bounded and converge to zero as $|x|\longrightarrow \infty$.
\end{theorem}
\begin{proof}
We prove this in two steps. The first step shows the boundedness of integrable solutions and the second step verifies the decay property.

\noindent{\it Part 1:} $u$ and $v$ are in $L^{\infty}(\mathbb{R}^n)$.

By exchanging the order of integration and choosing a suitably small $c>0$, we can write
\begin{align*}
u(x) \leq {} & C\Big( \int_{0}^{c} \frac{\int_{B_{t}(x)} |y|^{-\sigma_1}v(y)^q \,dy}{t^{n-\alpha}} \,\frac{dt}{t} +  \int_{c}^{\infty} \frac{\int_{B_{t}(x)} |y|^{-\sigma_1}v(y)^q \,dy}{t^{n-\alpha}} \,\frac{dt}{t} \Big) \\
\doteq {} & I_1 + I_2.
\end{align*} 
(i) We estimate $I_1$ first and assume $|x| > 1$, since the case where $|x|\leq 1$ can be treated similarly. Then H\"{o}lder's inequality yields
\begin{equation*}
\int_{B_{t}(x)}  \frac{v(y)^q}{|y|^{\sigma_1}} \,dy \leq Ct^{-\sigma_1} |B_{t}(x)|^{1-1/\ell}\|v^q\|_{\ell}
\end{equation*}
for $\ell > 1$. Then choose $\ell$ suitably large so that $\varepsilon q = 1/\ell$ is sufficiently small and so $v^q \in L^{\ell}(\mathbb{R}^n)$ as a result of Theorem \ref{integrability}. Thus,
\begin{equation*}
I_1 \leq C\|v^q\|_{\ell}\int_{0}^{c} \frac{|B_{t}(x)|^{1-\varepsilon q}}{t^{n-\alpha+\sigma_1}} \,\frac{dt}{t} \leq C\int_{0}^{c} t^{\alpha -\sigma_1 - nq\varepsilon} \,\frac{dt}{t} < \infty.
\end{equation*}
(ii) If $z \in B_{\delta}(x)$, then $B_{t}(x) \subset B_{t+\delta}(z)$. From this, observe that for $\delta \in (0,1)$ and $z \in B_{\delta}(x)$,
\begin{align*}
I_2 \leq {} & C\int_{c}^{\infty} \frac{\int_{B_{t}(x)} |y|^{-\sigma_1} v(y)^q\,dy}{t^{n-\alpha}} \,\frac{dt}{t} \\
\leq {} & C \int_{c}^{\infty} \frac{\int_{B_{t+\delta}(z)} |y|^{-\sigma_1} v(y)^q \,dy}{(t+\delta)^{n-\alpha}}\Big(\frac{t+\delta}{t}\Big)^{n-\alpha+1} \,\frac{d(t+\delta)}{t+\delta} \\
\leq {} & C(1+\delta)^{n-\alpha + 1}\int_{c+\delta}^{\infty} \frac{\int_{B_{t}(z)} |y|^{-\sigma_1} v(y)^q \,dy}{t^{n-\alpha}}\,\frac{dt}{t} \leq Cu(z).
\end{align*}
These estimates for $I_1$ and $I_2$ yield
$$ u(x) \leq C_{1} + C_{2}u(z) \,\text{ for }\, z \in B_{\delta}(x).$$
Integrating this estimate on $B_{\delta}(x)$ then applying H\"{o}lder's inequality gives us
\begin{align*}
u(x)\leq C_{1} + C_{2}|B_{\delta}(x)|^{-1}\int_{B_{\delta}(x)} u(z) \,dz \leq C_{1} + C_{2}|B_{\delta}(x)|^{-1/r_0}\|u\|_{r_0} < \infty.
\end{align*}
Hence, $u$ is bounded in $\mathbb{R}^n$. Using similar calculations, we can also show $v$ is bounded in $\mathbb{R}^n$.

\noindent{\it Part 2:} $u(x),v(x) \longrightarrow 0$ as $|x|\longrightarrow \infty$.

Choose $x \in \mathbb{R}^n$. Then for each $\varepsilon > 0$, there exists a sufficiently small $\delta > 0$ such that
\begin{equation*}
\int_{0}^{\delta} \frac{\int_{B_{t}(x)} |y|^{-\sigma_1}v(y)^q \,dy}{t^{n-\alpha}} \,\frac{dt}{t} \leq C\|v\|_{\infty}^q \int_{0}^{\delta} t^{\alpha-\sigma_1}\,
\frac{dt}{t} < \varepsilon.
\end{equation*} 
Likewise, for $|x-z|< \delta$, we have
\begin{align*}
\int_{\delta}^{\infty} {} & \frac{\int_{B_{t}(x)} |y|^{-\sigma_1}v(y)^q \,dy}{t^{n-\alpha}} \,\frac{dt}{t} \\
\leq {} & \int_{\delta}^{\infty} \frac{\int_{B_{t+\delta}(z)} |y|^{-\sigma_1}v(y)^q \,dy}{(t+\delta)^{n-\alpha}} \Big( \frac{t+\delta}{t} \Big)^{n-\alpha + 1} \,\frac{d(t+\delta)}{t+\delta} \\
\leq {} & C\int_{0}^{\infty} \frac{\int_{B_{t}(z)} |y|^{-\sigma_1} v(y)^q \,dy}{t^{n-\alpha}} \,\frac{dt}{t} = Cu(z).
\end{align*}
Therefore, the last two estimates imply
\begin{equation*}
u(x) \leq \varepsilon + Cu(z) \,\text{ for }\, z \in B_{\delta}(x),
\end{equation*}
and since $u \in L^{r_0}(\mathbb{R}^n)$, we obtain 
\begin{align*}
u(x)^{r_0} = \frac{1}{|B_{\delta}(x)|}\int_{B_{\delta}(x)} u(x)^{r_0} \,dz \leq C_{1}\varepsilon^{r_0} + C_{2}\frac{1}{|B_{\delta}(x)|}\int_{B_{\delta}(x)} u(z)^{r_0} \,dz \longrightarrow 0 
\end{align*}
as $|x| \longrightarrow \infty$ and $\varepsilon \longrightarrow 0$. Hence, $\lim_{|x|\longrightarrow \infty} u(x) = 0$, and similar calculations will show $\lim_{|x|\longrightarrow \infty} v(x) = 0$. This completes the proof.
\end{proof}

\subsection{Radial symmetry and monotonicity}

\begin{theorem}\label{radial symmetry}
Let $u,v$ be positive integrable solutions of \eqref{whls ie}. Then $u$ and $v$ are radially symmetric and monotone decreasing about the origin.
\end{theorem}
We employ the integral form of the method of moving planes (cf. \cite{CLO05,CLO06}) to prove this result, but first, we introduce some preliminary tools. For $\lambda \in\mathbb{R}$, define 
$$\Sigma_{\lambda} \doteq \{ x = (x_1,\ldots,x_n) \,|\, x_{1} \geq \lambda \},$$
let $x^{\lambda} = (2\lambda-x_1,x_2,\ldots,x_n)$ be the reflection of $x$ across the plane $\Gamma_{\lambda} \doteq \{ x_{1} = \lambda \}$ and let $u_{\lambda}(x) = u(x^{\lambda})$ and $v_{\lambda}(x) = v(x^{\lambda})$.

\begin{lemma}\label{lemma mp}
There holds
\begin{align*}
u_{\lambda}(x) - u(x) = {} & \int_{\Sigma_{\lambda}} \Big( \frac{1}{|x-y|^{n-\alpha}} - \frac{1}{|x^{\lambda} -y|^{n-\alpha}} \Big) \frac{1}{|y^{\lambda}|^{\sigma_1}}(v_{\lambda}(y)^q - v(y)^q)\,dy \\
- {} & \int_{\Sigma_{\lambda}} \Big( \frac{1}{|x-y|^{n-\alpha}} - \frac{1}{|x^{\lambda} -y|^{n-\alpha}} \Big) \Big(\frac{1}{|y|^{\sigma_1}} - \frac{1}{|y^{\lambda}|^{\sigma_1}}  \Big) v(y)^q\,dy; \\
v_{\lambda}(x) - v(x) = {} & \int_{\Sigma_{\lambda}} \Big( \frac{1}{|x-y|^{n-\alpha}} - \frac{1}{|x^{\lambda} -y|^{n-\alpha}} \Big) \frac{1}{|y^{\lambda}|^{\sigma_2}}(u_{\lambda}(y)^p - u(y)^p)\,dy \\
- {} & \int_{\Sigma_{\lambda}} \Big( \frac{1}{|x-y|^{n-\alpha}} - \frac{1}{|x^{\lambda} -y|^{n-\alpha}} \Big) \Big(\frac{1}{|y|^{\sigma_2}} - \frac{1}{|y^{\lambda}|^{\sigma_2}}  \Big) u(y)^p\,dy.
\end{align*}
\end{lemma}

\begin{proof}
We only prove the first identity since the second identity follows similarly. By noticing $|x-y^{\lambda}| = |x^{\lambda} - y|$, we obtain
\begin{align*}
u_{\lambda}(x) = {} & \int_{\Sigma_{\lambda}} \frac{v(y)^q}{|x^{\lambda}-y|^{n-\alpha}|y|^{\sigma_1}} \,dy + \int_{\Sigma_{\lambda}} \frac{v_{\lambda}(y)^q}{|x-y|^{n-\alpha}|y^{\lambda}|^{\sigma_1}} \,dy, \\
u(x) = {} & \int_{\Sigma_{\lambda}} \frac{v(y)^q}{|x-y|^{n-\alpha}|y|^{\sigma_1}} \,dy + \int_{\Sigma_{\lambda}} \frac{v_{\lambda}(y)^q}{|x^{\lambda}-y|^{n-\alpha}|y^{\lambda}|^{\sigma_1}} \,dy.
\end{align*}
Taking their difference yields
\begin{align*}
u_{\lambda}(x) - u(x) = {} & \int_{\Sigma_{\lambda}} \Big(\frac{v_{\lambda}(y)^q}{|x-y|^{n-\alpha}|y^{\lambda}|^{\sigma_1}} - \frac{v_{\lambda}(y)^q}{|x^{\lambda}-y|^{n-\alpha}|y^{\lambda}|^{\sigma_1}} \Big)\,dy \\
{} & - \int_{\Sigma_{\lambda}} \Big( \frac{v(y)^q}{|x-y|^{n-\alpha}|y|^{\sigma_1}} - \frac{v(y)^q}{|x^{\lambda}-y|^{n-\alpha}|y|^{\sigma_1}} \Big)\,dy \\
= {} & \int_{\Sigma_{\lambda}} \Big(\frac{1}{|x-y|^{n-\alpha}} - \frac{1}{|x^{\lambda}-y|^{n-\alpha}}\Big) \frac{v_{\lambda}(y)^q}{|y^{\lambda}|^{\sigma_1}} \,dy \\
{} & - \int_{\Sigma_{\lambda}} \Big( \frac{1}{|x-y|^{n-\alpha}} - \frac{1}{|x^{\lambda}-y|^{n-\alpha}}  \Big) \frac{ v(y)^q}{|y|^{\sigma_1}} \,dy,
\end{align*}
and the result follows accordingly. 
\end{proof}

\begin{proof}[Proof of Theorem \ref{radial symmetry}]
{\it Step 1:} We claim that there exists $N > 0$ such that if $\lambda \leq -N$, there hold
\begin{equation}\label{initial moving plane} 
u_{\lambda}(x) \leq u(x) \,\text{ and }\, v_{\lambda}(x) \leq v(x).
\end{equation}
Define $\Sigma_{\lambda}^{u} \doteq \{ x \in \Sigma_{\lambda} \,|\, u_{\lambda}(x) > u(x) \}$ and $\Sigma_{\lambda}^{v} \doteq \{ x \in \Sigma_{\lambda} \,|\, v_{\lambda}(x) > v(x) \}$. From Lemma \ref{lemma mp}, the mean-value theorem and since $$\Big( \frac{1}{|x-y|^{n-\alpha}} - \frac{1}{|x^{\lambda} -y|^{n-\alpha}} \Big) \Big(\frac{1}{|y|^{\sigma_1}} - \frac{1}{|y^{\lambda}|^{\sigma_1}}  \Big) \geq 0 \,\text{ for }\, y \in \Sigma_{\lambda},$$  
\begin{align*}
u_{\lambda}(x) - u(x) \leq {} & \int_{\Sigma_{\lambda}} \Big( \frac{1}{|x-y|^{n-\alpha}} - \frac{1}{|x^{\lambda} -y|^{n-\alpha}} \Big) \frac{1}{|y^{\lambda}|^{\sigma_1}}(v_{\lambda}(y)^q - v(y)^q)\,dy \\
\leq {} & \int_{\Sigma_{\lambda}^{v}} \frac{qv_{\lambda}(y)^{q-1}}{|x-y|^{n-\alpha}|y^{\lambda}|^{\sigma_1}} (v_{\lambda}(y) - v(y))\,dy.
\end{align*}
Thus, applying the weighted HLS inequality followed by H\"{o}lder's inequality gives us
\begin{align}\label{mpu}
\|u_{\lambda} - {} & u\|_{L^{r_0}(\Sigma_{\lambda}^{u})} \leq C\|v_{\lambda}^{q-1}(v_{\lambda}-v)\|_{L^{\frac{n r_0}{n+r_{0}(\alpha - \sigma_1)}}(\Sigma_{\lambda}^{v})} \notag \\
\leq {} & C\|v_{\lambda}\|_{L^{s_0}(\Sigma_{\lambda}^{v})}^{q-1}\|v_{\lambda}-v\|_{L^{s_0}(\Sigma_{\lambda}^{v})} 
\leq C\|v\|_{L^{s_0}(\Sigma_{\lambda}^{C})}^{q-1}\|v_{\lambda}-v\|_{L^{s_0}(\Sigma_{\lambda}^{v})}.
\end{align}
Similarly, there holds
\begin{align}\label{mpv}
\|v_{\lambda} - {} & v\|_{L^{s_0}(\Sigma_{\lambda}^{v})} \leq C\|u_{\lambda}^{p-1}(u_{\lambda} - u)\|_{L^{\frac{n s_0}{n+s_{0}(\alpha - \sigma_2)}}(\Sigma_{\lambda}^{u})} \notag \\
\leq {} & C\|u_{\lambda}\|_{L^{r_0}(\Sigma_{\lambda}^{u})}^{p-1}\|u_{\lambda}-u\|_{L^{r_0}(\Sigma_{\lambda}^{u})}
\leq C\|u\|_{L^{r_0}(\Sigma_{\lambda}^{C})}^{p-1}\|u_{\lambda}-u\|_{L^{r_0}(\Sigma_{\lambda}^{u})}.
\end{align}
By virtue of $(u,v) \in L^{r_0}(\mathbb{R}^n) \times L^{s_0}(\mathbb{R}^n)$, we can choose $N$ suitably large such that for $\lambda \leq -N$, there holds
$$ C ^2\|u\|_{L^{r_0}(\Sigma_{\lambda}^{C})}^{p-1} \|v\|_{L^{s_0}(\Sigma_{\lambda}^{C})}^{q-1} \leq 1/2.$$
Therefore, combining \eqref{mpu} with \eqref{mpv} gives us
\begin{equation*}
  \left\{\begin{array}{cl}
    \|u_{\lambda} - u\|_{L^{r_0}(\Sigma_{\lambda}^{u})} \leq \frac{1}{2}\|u_{\lambda} - u\|_{L^{r_0}(\Sigma_{\lambda}^{u})},\medskip \\ 
    \|v_{\lambda} - v\|_{L^{s_0}(\Sigma_{\lambda}^{v})} \leq \frac{1}{2}\|v_{\lambda} - v\|_{L^{s_0}(\Sigma_{\lambda}^{v})}, \\
  \end{array}
\right.
\end{equation*}
which further implies that $\|u_{\lambda} - u\|_{L^{r_0}(\Sigma_{\lambda}^{u})} = 0$ and  $\|v_{\lambda} - v\|_{L^{s_0}(\Sigma_{\lambda}^{v})} = 0$. Thus, both $\Sigma_{\lambda}^{u}$ and $\Sigma_{\lambda}^{v}$ have measure zero and are therefore empty. This concludes the proof of the first claim.

{\it Step 2:} We can move the plane $\Gamma_{\lambda}$ to the right provided that \eqref{initial moving plane} holds. Let
$$ \lambda_0 \doteq \sup \lbrace \lambda \,|\, u_{\lambda}(x) \leq u(x), v_{\lambda}(x) \leq v(x)\rbrace.$$
We claim that $u,v$ are symmetric about the plane $\Gamma_{\lambda_0}$, i.e.
$$ u_{\lambda_0}(x) = u(x) \,\text{ and }\, v_{\lambda_0}(x) = v(x) \,\text{ for all }\, x \in \Sigma_{\lambda_0}. $$
On the contrary, assume $\lambda_0 \leq 0$ and for all $x \in \Sigma_{\lambda_0}$,
$$ u_{\lambda_0}(x) \leq u(x) \,\text{ and }\, v_{\lambda_0}(x) \leq v(x), \,\text{ but }\, u_{\lambda_0}(x) \not\equiv u(x) \,\text{ or }\, v_{\lambda_0}(x) \not\equiv v(x).$$
But we will show this implies the plane can be moved further to the right, thereby contradicting the definition of $\lambda_0$. Namely, there is a small $\varepsilon > 0$ such that
\begin{equation}\label{contradict mp}
u_{\lambda}(x) \leq u(x) \,\text{ and }\, v_{\lambda}(x) \leq v(x), \,~\, x \in\Sigma_{\lambda}, \,\text{ for all }\, \lambda \in [\lambda_0,\lambda_0 + \varepsilon).
\end{equation}
In the case, say $v_{\lambda_0}(x) \not\equiv v(x)$ on $\Sigma_{\lambda_0}$, Lemma \ref{lemma mp} indicates $u_{\lambda_0}(x) < u(x)$ in the interior of $\Sigma_{\lambda_0}$. Define
$$ \Phi_{\lambda_0}^{u} \doteq \lbrace x \in \Sigma_{\lambda_0} \,|\, u_{\lambda_0}(x) \geq u(x)\rbrace \,\text{ and }\, \Phi_{\lambda_0}^{v} \doteq \lbrace x \in \Sigma_{\lambda_0} \,|\, v_{\lambda_0}(x) \geq v(x)\rbrace.$$
Then $\Phi_{\lambda_0}^{u}$ and $\Phi_{\lambda_0}^{v}$ have measure zero and 
$$ \lim_{\lambda \rightarrow \lambda_0 } \Sigma_{\lambda}^{u} \subset \Phi_{\lambda_0}^{u} \,\text{ and }\, \lim_{\lambda \rightarrow \lambda_0 } \Sigma_{\lambda}^{v} \subset \Phi_{\lambda_0}^{v}. $$
Let $\Omega^{\ast}$ denote the reflection of the set $\Omega$ about the plane $\Gamma_{\lambda}$. According to the integrability of $u$ and $v$, we can choose a suitably small $\varepsilon$ such that for all $\lambda \in [\lambda_0,\lambda_0 + \varepsilon)$, $\|u\|_{L^{r_0}((\Sigma_{\lambda}^{u})^{\ast})}$ and $\|v\|_{L^{s_0}((\Sigma_{\lambda}^{v})^{\ast})}$ are sufficiently small. Therefore, \eqref{mpu} and \eqref{mpv} imply 
\begin{equation*}
  \left\{\begin{array}{cl}
	\|u_{\lambda} - u\|_{L^{r_0}(\Sigma_{\lambda}^{u})} \leq C\|v\|_{L^{s_0}((\Sigma_{\lambda}^{v})^{\ast})}^{q-1}\|v_{\lambda}-v\|_{L^{s_0}	(\Sigma_{\lambda}^{v})} \leq \frac{1}{2}\|v_{\lambda}-v\|_{L^{s_0}	(\Sigma_{\lambda}^{v})},\medskip \\ 
    \|v_{\lambda} - v\|_{L^{s_0}(\Sigma_{\lambda}^{v})} \leq C\|u\|_{L^{r_0}((\Sigma_{\lambda}^{u})^{\ast})}^{p-1}\|u_{\lambda}-u\|_{L^{r_0}(\Sigma_{\lambda}^{u})} \leq \frac{1}{2}\|u_{\lambda}-u\|_{L^{r_0}(\Sigma_{\lambda}^{u})}, \\
  \end{array}
\right.
\end{equation*}
and we deduce that $\Sigma_{\lambda}^{u}$ and $\Sigma_{\lambda}^{v}$ are both empty. This proves \eqref{contradict mp}. Hence, we conclude that $u$ and $v$ are symmetric and decreasing about the plane $\Gamma_{\lambda_0}$.

{\it Step 3:} We assert that $u$ and $v$ are radially symmetric and decreasing about the origin.

First, notice that $\lambda_0$ is indeed equal to zero. Otherwise, if $\lambda_0 < 0$, then Lemma \ref{lemma mp} yields
\begin{align*}
0 = {} & u(x) - u_{\lambda_0}(x) \\
= {} & \int_{\Sigma_{\lambda_0}} \Big( \frac{1}{|x-y|^{n-\alpha}} - \frac{1}{|x^{\lambda_0} -y|^{n-\alpha}} \Big) \Big(\frac{1}{|y|^{\sigma_1}} - \frac{1}{|y^{\lambda_0}|^{\sigma_1}}  \Big) v(y)^q\,dy \not\equiv 0,
\end{align*}
which is impossible. Therefore, $u$ and $v$ are symmetric and decreasing about the plane $\Gamma_{\lambda_0 = 0}$, and since the coordinate direction $x_1$ can be chosen arbitrarily, we conclude that $u$ and $v$ must be radially symmetric and decreasing about the origin. This completes the proof of the theorem.
\end{proof}

\section{Fast decay rates of integrable solutions}\label{fast section}
In this section, $u,v$ are taken to be positive integrable solutions of system \eqref{whls ie} unless further specified.

\subsection{Fast decay rate for $u(x)$}

\begin{proposition}\label{fast decay of u}
There holds the following.
\begin{enumerate}[(i)]
\item The improper integral, $\ds A_0 \doteq \int_{\mathbb{R}^n}  \frac{v(y)^q}{|y|^{\sigma_1}} \,dy < \infty$, is convergent;

\item $\ds \lim_{|x|\longrightarrow \infty} u(x)|x|^{n-\alpha} = A_0 .$
\end{enumerate}
\end{proposition}

\begin{remark}
According to this, we can find a large $R>0$ such that
\begin{equation}\label{long u}
u(x) = \frac{A_0 + \mathrm{o}(1)}{|x|^{n-\alpha}} \,~\,\text{ for }\, x \in \mathbb{R}^n \backslash B_{R}(0),
\end{equation}
and we shall often invoke this property in establishing the decay rates for $v(x)$.
\end{remark}

\begin{proof}
(i) Without loss of generality, we assume $\sigma_{1} > 0$ since the proof for the unweighted case is similar but far simpler. For each $R> 0$, since $v \in L^{\infty}(\mathbb{R}^n)$ and $\sigma_1 < n$, we have
\begin{equation*}
\int_{B_{R}(0)} \frac{v(y)^q}{|y|^{\sigma_1}} \,dy < \infty.
\end{equation*}
So it remains to show $\int_{B_{R}(0)^C} \frac{v(y)^q}{|y|^{\sigma_1}} \,dy < \infty$. There are two cases to consider. (1.) Assume $n-\alpha \leq p(n-\alpha)-(\alpha-\sigma_2)$. It is clear that $q\geq p$, $\sigma_1 \geq \sigma_2$ and \eqref{non-subcritical} imply
$ q \geq \frac{n+\alpha - 2\sigma_1}{n-\alpha}$. Choose $\varepsilon > 0$ with $\varepsilon \in (\alpha - 2\sigma_1,\alpha - \sigma_1)$. Then set $\ell = \frac{n+\varepsilon}{n+\alpha-2\sigma_1}$ and $\ell' = \frac{n+\varepsilon}{\varepsilon - \alpha + 2\sigma_1}$ so that $\frac{1}{\ell} + \frac{1}{\ell'} = 1$, $lq > \frac{n}{n-\alpha}$, and $\ell' > \frac{n}{\sigma_1}$. Therefore, H\"{o}lder's inequality and Theorem \ref{integrability} imply
\begin{align*}
\int_{B_{R}(0)^{C}} \frac{v(y)^q}{|y|^{\sigma_1}} \,dy \leq {} & \Big( \int_{B_{R}(0)^{C}} \frac{1}{|y|^{\sigma_1 \frac{\ell}{\ell-1}}} \,dy \Big)^{\frac{\ell-1}{\ell}} \Big(\int_{B_{R}(0)^{C}} v(y)^{\ell q} \,dy\Big)^{1/\ell} \\
\leq {} & C\Big( \int_{R}^{\infty} t^{n-\sigma_1 \ell'} \,\frac{dt}{t}\Big)^{\frac{1}{\ell'}}  \Big(\int_{B_{R}(0)^{C}} v(y)^{\ell q} \,dy\Big)^{1/\ell} \\
\leq {} & C\|v\|_{\ell q}^{q} < \infty.
\end{align*}
(2.) Assume $n-\alpha > p(n-\alpha)-(\alpha-\sigma_2)$. For small $\varepsilon > 0$, take $\ell = \frac{n}{n-\sigma_1 + \varepsilon}$ and $\ell' = \frac{n}{\sigma_1 - \varepsilon}$ so that $\frac{1}{\ell} + \frac{1}{\ell'} = 1$. From the non-subcritical condition \eqref{non-subcritical} and since $pq > 1$, we get
\begin{equation}\label{consequence of nonsub}
\frac{q(n-\sigma_2)+(n-\sigma_1)}{q(1+p)} = \frac{n-\sigma_1}{q(1+p)} + \frac{n-\sigma_2}{1+p} < \frac{n-\sigma_1}{1+q} + \frac{n-\sigma_2}{1+p} \leq n-\alpha.
\end{equation}
Thus, $\frac{n-\sigma_1}{q(1+p)} < n-\alpha - \frac{n-\sigma_2}{1+p} \,\Longrightarrow\, \frac{n-\sigma_1}{q} < p(n-\alpha) - (\alpha-\sigma_2)$. This yields $\frac{n-\sigma_1 + \varepsilon}{nq} < \frac{p(n-\alpha)-(\alpha-\sigma_2)}{n}$ for a sufficiently small $\varepsilon$, which implies
$$\frac{1}{\ell q} < \frac{p(n-\alpha) - (\alpha-\sigma_2)}{n} \,\text{ and }\, \ell' > \frac{n}{\sigma_1}.$$ 
Hence, H\"{o}lder's inequality and Theorem \ref{integrability} imply 
$$\int_{B_{R}(0)^C} \frac{v(y)^q}{|y|^{\sigma_1}} \,dy \leq C\|v\|_{lq}^{q} < \infty.$$

(ii) For fixed $R>0$, write
\begin{align*}
\int_{\mathbb{R}^n} {} & \frac{|x|^{n-\alpha}v(y)^q}{|x-y|^{n-\alpha}|y|^{\sigma_1}} \,dy = \int_{B_{R}(0)} \frac{|x|^{n-\alpha}v(y)^q}{|x-y|^{n-\alpha}|y|^{\sigma_1}} \,dy \\
{} & + \int_{\lbrace\mathbb{R}^n \backslash B_{R}(0)\rbrace \backslash B_{|x|/2}(x)} \frac{|x|^{n-\alpha}v(y)^q}{|x-y|^{n-\alpha}|y|^{\sigma_1}} \,dy + \int_{B_{|x|/2}(x)} \frac{|x|^{n-\alpha}v(y)^q}{|x-y|^{n-\alpha}|y|^{\sigma_1}} \,dy \\
\doteq {} & J_1 + J_2 + J_3.
\end{align*}
Set $$J_{1}' \doteq \int_{B_{R}(0)} \frac{v(y)^q}{|y|^{\sigma_1}} \Big( \frac{|x|^{n-\alpha}}{|x-y|^{n-\alpha}} - 1 \Big) \,dy.$$
For $y\in B_{R}(0)$ and large $|x|$,
$$  \frac{v(y)^q}{|y|^{\sigma_1}} \Big| \frac{|x|^{n-\alpha}}{|x-y|^{n-\alpha}} - 1 \Big| \leq C\frac{v(y)^q}{|y|^{\sigma_1}} \in L^{1}(\mathbb{R}^n) $$
as a result of part (i). By virtue of the Lebesgue dominated convergence theorem, $|J_{1}'|\longrightarrow 0$ as $|x| \longrightarrow \infty$, which implies
\begin{equation}\label{J1}
\lim_{R\longrightarrow \infty} \lim_{|x|\longrightarrow \infty} J_{1} = A_0.
\end{equation}
Next, notice that if $y \in \lbrace\mathbb{R}^n \backslash B_{R}(0)\rbrace\backslash B_{|x|/2}(x)$, then $|x-y| \geq |x|/2$. Therefore, 
\begin{equation}\label{J2}
J_{2} \leq C\int_{\mathbb{R}^n \backslash B_{R}(0)} \frac{v(y)^q}{|y|^{\sigma_1}}\,dy \longrightarrow 0 \,\text{ as }\, R \longrightarrow \infty.
\end{equation}
Set 
$$J_{3}' \doteq \frac{J_3}{|x|^{n-\alpha}} = \int_{B_{|x|/2}(x)} \frac{v(y)^q}{|x-y|^{n-\alpha}|y|^{\sigma_1}}\,dy.$$
In view of Theorem \ref{radial symmetry}, $v$ is radially symmetric and decreasing about the origin. Therefore,
\begin{equation}\label{bound J3'}
J_{3}' \leq v(|x|/2)^q \int_{B_{|x|/2}(x)} \frac{dy}{|x-y|^{n-\alpha}|y|^{\sigma_1}} \leq Cv(|x|/2)^q |x|^{\alpha - \sigma_1}.
\end{equation}
By Theorem \ref{integrability}, $v \in L^{s}(\mathbb{R}^n)$ such that
\begin{equation*}
\frac{1}{s} = \min\Big\lbrace \frac{n-\alpha}{n}, \frac{p(n-\alpha) - (\alpha-\sigma_2)}{n} \Big\rbrace - \frac{\varepsilon}{n} ~\,\text{ for sufficiently small }\, \varepsilon > 0.
\end{equation*}
Then, combining this with the decreasing property of $v$ yields
\begin{equation}\label{s integrability}
v(|x|/2)^s |x|^{n} \leq C \int_{B_{|x|/2}(0) \backslash B_{|x|/4}(0) } v(y)^s \,dy < \infty.
\end{equation}
We claim that
\begin{equation}\label{J3'}
|x|^{n-\alpha}J_{3}' = \mathrm{o}(1) \,\text{ as }\, |x| \longrightarrow \infty.
\end{equation}
To do so, we consider two cases. 

Case 1. Let $n-\alpha \leq p(n-\alpha)-(\alpha-\sigma_2)$. Then \eqref{s integrability} implies that
$$v(|x|/2)^q |x|^{q(n-\alpha-\varepsilon)} \leq C,$$
which when combined with \eqref{bound J3'}, yields
$$ |x|^{q(n-\alpha - \varepsilon) - (\alpha-\sigma_1)} J_{3'} \leq Cv(|x|/2)^q |x|^{q(n-\alpha - \varepsilon)} \leq C. $$
Recall that $q \geq \frac{n+\alpha-2\sigma_1}{n-\alpha}$, which implies that $q(n-\alpha) - (\alpha-\sigma_1) \geq n-\sigma_1 > n-\alpha$. Thus, by choosing $\varepsilon$ suitably small and sending $|x|\longrightarrow \infty$ in the previous estimate after the appropriate calculations, we obtain \eqref{J3'}.

Case 2. Let $n-\alpha > p(n-\alpha)-(\alpha-\sigma_2)$. Then \eqref{s integrability} implies
$$v(|x|/2)^q |x|^{q(p(n-\alpha)-(\alpha-\sigma_2))} \leq C,$$
which, when combined with \eqref{bound J3'}, gives us
$$ |x|^{q[p(n-\alpha)-(\alpha-\sigma_2)]-(\alpha-\sigma_1)}J_{3'} \leq C. $$
It is easy to check that \eqref{consequence of nonsub} implies that
$$ q[p(n-\alpha)-(\alpha-\sigma_2)] - (\alpha - \sigma_1) > n-\alpha.$$
Assertion \eqref{J3'} follows by sending $|x|\longrightarrow \infty$ in the last estimate after the appropriate calculations. 

Notice that \eqref{J3'} implies that
\begin{equation}\label{J3}
\lim_{|x|\longrightarrow \infty} J_{3} = 0. 
\end{equation}
Hence, \eqref{J1},\eqref{J2} and \eqref{J3} imply $$\lim_{|x|\longrightarrow \infty} |x|^{n-\alpha}u(x) = A_0,$$
and this completes the proof of the proposition.
\end{proof}

\subsection{Fast decay rates for $v(x)$}

\begin{proposition}\label{fast decay of v 1}
If $p(n-\alpha) + \sigma_2 > n$, then 
\begin{enumerate}[(i)]
\item $\ds A_1 \doteq \int_{\mathbb{R}^n} \frac{u(y)^p}{|x|^{\sigma_2}} \,dy < \infty$;
\item $\ds\lim_{|x|\longrightarrow \infty} |x|^{n-\alpha}v(x) = A_1. $
\end{enumerate}
\end{proposition}

\begin{proof}
Since $p > \frac{n-\sigma_2}{n-\alpha}$, (i) follows from Theorem \ref{integrability} and H\"{o}lder's inequality similar to the proof of Proposition \ref{fast decay of u}(i). 

To prove (ii), write
\begin{align*}
|x|^{n-\alpha}v(x) = {} & \int_{B_{R}(0)} \frac{|x|^{n-\alpha}u(y)^p}{|x-y|^{n-\alpha}|y|^{\sigma_2}}\,dy \\
{} & + \int_{\mathbb{R}^n \backslash B_{R}(0)}  \frac{ |x|^{n-\alpha} [A_0 + \mathrm{o}(1)]^p }{|x-y|^{n-\alpha}|y|^{p(n-\alpha) + \sigma_2}} \,dy \doteq J_4 + J_5.
\end{align*}
For a large $R>0$, if $y\in B_{R}(0)$, then $\lim_{|x|\longrightarrow \infty} \frac{|x|^{n-\alpha}}{|x-y|^{n-\alpha}} = 1$, and the Lebesgue dominated convergence theorem implies that
\begin{equation}\label{J6}
\lim_{R\longrightarrow\infty}\lim_{|x|\longrightarrow \infty} J_{4} = \lim_{R\longrightarrow\infty}\lim_{|x|\longrightarrow \infty} \int_{B_{R}(0)} \frac{|x|^{n-\alpha} u(y)^p}{|x-y|^{n-\alpha}|y|^{\sigma_2}} \,dy = A_1.
\end{equation}
Likewise, since $p(n-\alpha) + \sigma_2 > n$, 
\begin{equation*}
J_{5} \leq C\int_{R}^{\infty} t^{n-p(n-\alpha) - \sigma_2} \,\frac{dt}{t} = \mathrm{o}(1) \,\text{ as }\, R \longrightarrow \infty \,\text{ for suitably large }\, |x|.
\end{equation*}
Particularly,
\begin{equation*}
\int_{\mathbb{R}^n \backslash B_{R}(0)} \frac{ |x|^{n-\alpha}}{|x-y|^{n-\alpha}|y|^{p(n-\alpha)+\sigma_2}}\,dy = \mathrm{o}(1) \,\text{ as }\, |x| \longrightarrow \infty,\, R \longrightarrow \infty.
\end{equation*}
Hence, assertion (ii) follows from this and \eqref{J6}.
\end{proof}

\begin{proposition}\label{fast decay of v 2}
If $p(n-\alpha) + \sigma_2 = n$, then $$\lim_{|x|\longrightarrow\infty} \frac{|x|^{n-\alpha}}{\ln |x|}v(x) = A_{0}^p |S^{n-1}|,$$
where $|S^{n-1}|$ is the surface area of the $(n-1)$-dimensional unit sphere. 
\end{proposition}

\begin{proof}
From \eqref{long u}, we have for large $R>0$,
\begin{align*}
\frac{|x|^{n-\alpha}}{\ln |x|}v(x) = {} & \frac{1}{\ln |x|}\int_{B_{R}(0)} \frac{|x|^{n-\alpha}u(y)^p}{|x-y|^{n-\alpha}|y|^{\sigma_2}} \,dy \\
{} & + \frac{(A_0 + \mathrm{o}(1))^p}{\ln |x|}\int_{\mathbb{R}^n\backslash B_{R}(0)} \frac{|x|^{n-\alpha}}{|x-y|^{n-\alpha}|y|^n} \,dy. 
\end{align*}  
Indeed, \eqref{J6} implies that
\begin{equation*}
\frac{1}{\ln |x|}\int_{B_{R}(0)} \frac{|x|^{n-\alpha}u(y)^p}{|x-y|^{n-\alpha}|y|^{\sigma_2}} \,dy = \mathrm{o}(1) \,\text{ as }\, |x| \longrightarrow \infty.
\end{equation*}
Thus, it only remains to show that
\begin{equation}\label{second case term}
\frac{1}{\ln |x|}\int_{\mathbb{R}^n\backslash B_{R}(0)} \frac{|x|^{n-\alpha}}{|x-y|^{n-\alpha}|y|^n} \,dy \longrightarrow |S^{n-1}| \,\text{ as }\, |x| \longrightarrow \infty.
\end{equation}
Indeed, for large $R>0$ and $c \in (0,1/2)$, polar coordinates and a change of variables give us
\begin{align*}
\frac{1}{\ln |x|}\int_{\mathbb{R}^n \backslash B_{R}(0)} \frac{|x|^{n-\alpha}}{|x-y|^{n-\alpha}|y|^n} \,dy = {} & \frac{1}{\ln |x|}\int_{\frac{R}{|x|}}^{c}\int_{S^{n-1}} \frac{1}{r|e-rw|^{n-\alpha}} \,dS_{w} dr \\
{} & + \frac{1}{\ln |x|}\int_{\mathbb{R}^n \backslash B_{c}(0)} \frac{1}{|z|^n |e-z|^{n-\alpha}} \,dz,
\end{align*}
where $e$ is a unit vector in $\mathbb{R}^n$. Clearly, the integral in the second term can be bounded above by a positive constant depending only on $c$, since $n-\alpha < n$ for $z$ near $e$ and $n-\alpha + n > n$ near infinity. Thus, we deduce that
$$  \frac{1}{\ln |x|}\int_{\mathbb{R}^n \backslash B_{c}(0)} \frac{1}{|z|^n |e-z|^{n-\alpha}} \,dz = \mathrm{o}(1) \,\text{ as }\, |x| \longrightarrow \infty.$$
For $r \in (0,c)$, it is also clear that $1-c \leq |e-rw|\leq 1+c$. Therefore, there exists $\theta \in (-1,1)$ such that $|e-rw|= 1 + \theta c$, which leads to  
\begin{align*}
\frac{1}{\ln |x|}\int_{\frac{R}{|x|}}^{c} \int_{S^{n-1}} \frac{1}{r|e-rw|^{n-\alpha}}\,dw dr 
= {} & \frac{|S^{n-1}|}{(1+\theta c)^{n-\alpha}\ln |x|} (\ln c - \ln R + \ln |x|) \\
\longrightarrow {} & \frac{|S^{n-1}|}{(1+\theta c)^{n-\alpha}} \,\text{ as }\, |x| \longrightarrow \infty.
\end{align*}
Hence, by sending $c \longrightarrow 0$, we obtain \eqref{second case term} and this concludes the proof.
\end{proof}

\begin{proposition}\label{fast decay of v 3}
If $p(n-\alpha) + \sigma_2 < n$, then 
$$ A_{2} \doteq A_{0}^{p}\int_{\mathbb{R}^n} \frac{dz}{|z|^{p(n-\alpha)+\sigma_2}|e-z|^{n-\alpha}} < \infty.$$
Moreover,
$$ \lim_{|x|\longrightarrow \infty} |x|^{p(n-\alpha) - (\alpha - \sigma_2)}v(x) = A_{2}.$$
\end{proposition}

\begin{proof}
According to the non-subcritical condition, we have that $\frac{n-\sigma_2}{1+p} < n-\alpha$. Therefore, the integrand in $A_2$ decays with the following rates: $(1+p)(n-\alpha) + \sigma_2 > n$ near infinity; $n-\alpha < n$ near $e$; and $p(n-\alpha)+\sigma_2 < n$ near the origin. Thus, we conclude that $A_2 < \infty$.

For large $R>0$, we use \eqref{long u} to write
\begin{align}\label{third case}
|x|^{p(n-\alpha) - (\alpha-\sigma_2)} {} & v(x) = |x|^{p(n-\alpha) + \sigma_2 - n}\int_{B_{R}(0)} \frac{|x|^{n-\alpha}u(y)^p}{|x-y|^{n-\alpha}|y|^{\sigma_2}} \,dy \notag \\
{} & + (A_0 + \mathrm{o}(1))^p \int_{\mathbb{R}^n \backslash B_{R}(0)} \frac{|x|^{(p+1)(n-\alpha)+\sigma_2 - n}}{|x-y|^{n-\alpha}|y|^{p(n-\alpha)+\sigma_2}} \,dy.
\end{align}
Indeed, if $y\in B_{R}(0)$, then
\begin{equation*}
|x|^{p(n-\alpha)+ \sigma_2 - n}\int_{B_{R}(0)} \frac{|x|^{n-\alpha}u(y)^p}{|x-y|^{n-\alpha}|y|^{\sigma_2}} \,dy 
\leq C|x|^{p(n-\alpha) + \sigma_2 - n} \longrightarrow 0
\end{equation*}
as $|x|\longrightarrow \infty$. Likewise, as $|x|\longrightarrow \infty$
\begin{align*}
\int_{\mathbb{R}^n\backslash B_{R}(0)} \frac{|x|^{(p+1)(n-\alpha) + \sigma_2 - n}}{|x-y|^{n-\alpha}|y|^{p(n-\alpha)+\sigma_2}} \,dy 
= {} & \int_{\mathbb{R}^n\backslash B_{R/|x|}(0)} \frac{dz}{|z|^{p(n-\alpha) + \sigma_2}|e-z|^{n-\alpha}} \\
\longrightarrow {} & \frac{A_2}{A_{0}^{p}}.
\end{align*}
Inserting these calculations into \eqref{third case} leads to the desired result.
\end{proof}

\subsection{Characterization of integrable solutions}
\begin{proposition}\label{bounded decay are integrable}
Let $u,v$ be positive solutions of \eqref{whls ie} satisfying \eqref{non-subcritical0}. If $u,v$ are bounded and decay with the fast rates as $|x|\longrightarrow\infty$, then $u,v$ are integrable solutions.
\end{proposition}

\begin{proof}
Suppose $u,v$ are bounded and decay with the fast rates. From \eqref{non-subcritical0}, it is clear that $(n-\alpha)r_0 > n$. Thus,
\begin{equation*}
\int_{\mathbb{R}^n} u(x)^{r_0} \,dx \leq C + \int_{\mathbb{R}^n\backslash B_{R}(0)} u(x)^{r_0}\,dx \leq C_{1} + C_{2}\int_{R}^{\infty} t^{n-(n-\alpha)r_0} \,\frac{dt}{t} < \infty.
\end{equation*}
Similarly, $\int_{\mathbb{R}^n} v(x)^{s_0}\,dx < \infty$ if $v$ decays with rate $|x|^{-(n-\alpha)}$. If $v$ decays with the rate $|x|^{-(n-\alpha)}\ln |x|$, then we can find a suitably large $R>0$ and small $\varepsilon > 0$ for which $(\ln |x|)^{s_0} \leq |x|^{\varepsilon}$ for $|x| > R$. Then, we also get $n-(n-\alpha)s_0 + \varepsilon < 0$ provided $\varepsilon$ is sufficiently small and this implies
\begin{equation*}
\int_{\mathbb{R}^n} v(x)^{s_0} \,dx \leq C_{1} + C_{2} \int_{R}^{\infty} t^{n-(n-\alpha)s_0 + \varepsilon} \,\frac{dt}{t} < \infty.
\end{equation*}
If $v$ decays with the rate $|x|^{-(p(n-\alpha)-(\alpha-\sigma_2))}$, then \eqref{non-subcritical0} implies that $q_0 < n-\alpha$, which further yields $p q_0 - \alpha + \sigma_2 < p(n-\alpha) - (\alpha - \sigma_2)$. From this we deduce that $n - (p(n-\alpha) - (\alpha-\sigma_2))s_0 < 0$. Therefore,
\begin{equation*}
\int_{\mathbb{R}^n} v(x)^{s_0} \,dx \leq C_{1} + C_{2} \int_{R}^{\infty} t^{n - (p(n-\alpha) - (\alpha-\sigma_2))s_0} \,\frac{dt}{t} < \infty.
\end{equation*}
In any case, we conclude that $(u,v) \in L^{r_0}(\mathbb{R}^n)\times L^{s_0}(\mathbb{R}^n)$.
\end{proof}

\begin{proof}[Proof of Theorem \ref{integrable theorem}]
Propositions \ref{fast decay of u}--\ref{bounded decay are integrable} show $u,v$ are positive integrable solutions if and only if they are bounded and decay with the fast rates as $|x|\longrightarrow\infty$. Lastly, it remains to show that \eqref{whls ie} does not admit any positive integrable solution in the supercritical case. To prove this, assume $u$ and $v$ are positive integrable solutions. Then, we can apply similar arguments found in the proof of Proposition \ref{fast decay of u} to show that
$$ \int_{\mathbb{R}^n} \frac{u(x)^{p+1}}{|x|^{\sigma_2}} \,dx = \int_{\mathbb{R}^n} \frac{v(x)^{q+1}}{|x|^{\sigma_1}} \,dx < \infty.$$
Then, as in the proof of Theorem \ref{Liouville}, we can deduce the same Pohozaev type identity to arrive at
\begin{equation*}
\Bigg\lbrace \frac{n-\sigma_1}{1+q} + \frac{n-\sigma_2}{1+p} - (n-\alpha) \Bigg\rbrace \int_{\mathbb{R}^n} \frac{v(x)^{q+1}}{|x|^{\sigma_1}}\,dx 
= 0,
\end{equation*}
but this contradicts the supercritical condition. This completes the proof of the theorem.
\end{proof}

\section{Asymptotic properties of non-integrable solutions}\label{non-integrable section}
In this section, we assume $u,v$ are bounded positive solutions of system \eqref{whls ie}.
\begin{proposition}\label{nonint 1}
Let $\theta_1 < q_0$ and $\theta_2 < p_0$. Then there does not exist a positive constant $c$ such that either
$$ u(x) \geq c(1+|x|)^{-\theta_1} \,\text{ or }\, v(x) \geq  c(1+|x|)^{-\theta_2}.$$
\end{proposition}

\begin{proof}
Assume that there exists such a $c>0$ in which,
$$ u(x) \geq c(1+|x|)^{-\theta_1} \,\text{ where }\, \theta_1 < q_0.$$
Then there holds for large $x$,
$$v(x) \geq \int_{B_{|x|/2}(x)} \frac{u(y)^p}{|x-y|^{n-\alpha}|y|^{\sigma_2}} \,dy \geq  c(1+|x|)^{-a_1},$$
where $b_0 = \theta_1$ and $a_1 = pb_0 - \alpha + \sigma_2$. Thus, inserting this into the first integral equation yields
$$ u(x) \geq \int_{B_{|x|/2}(x)} \frac{v(y)^q}{|x-y|^{n-\alpha}|y|^{\sigma_1}} \,dy \geq c(1+|x|)^{-b_1},$$
where $b_1 = q a_1 - \alpha + \sigma_1$. By inductively repeating this argument, we arrive at
$$ v(x) \geq c(1+|x|)^{-a_j} \,\text{ and }\, u(x)\geq c(1+|x|)^{-b_j} , $$
where 
$$ a_{j+1} = pb_{j} - \alpha + \sigma_2 \,\text{ and }\, b_{j} = qa_j - \alpha + \sigma_1 \,\text{ for }\, j = 1,2,3,\ldots.$$
A simple calculation yields
\begin{align*}
b_{k} = {} & q a_k - \alpha + \sigma_1 = q(pb_{k-1} - \alpha + \sigma_2) - \alpha + \sigma_1 \\
= {} & pq b_{k-1} - (\alpha(1+q) - (\sigma_1 + \sigma_2 q)) \\
= {} & (pq)^2 b_{k-2} - (\alpha(1+q) - (\sigma_1 + \sigma_2 q))(1+pq) \\
\vdots \\
= {} & (pq)^k b_{0} - (\alpha(1+q) - (\sigma_1 + \sigma_2 q))(1+pq + (pq)^2 + \ldots + (pq)^{k-1}) \\
= {} & (pq)^k b_{0} - (\alpha(1+q) - (\sigma_1 + \sigma_2 q))\frac{(pq)^k - 1}{pq - 1} = (pq)^k (b_{0} - q_0) + q_0.
\end{align*}
Since $pq > 1$ and $b_0 = \theta_1 < q_0$, we can find a sufficiently large $k_0$ such that $b_{k_0} < 0$, but then this implies that for a suitable choice of $R>0$,
\begin{align*}
v(x) \geq c\int_{\mathbb{R}^n \backslash B_{R}(0)} \frac{u(y)^p}{|x-y|^{n-\alpha}|y|^{\sigma_2}}\,dy \geq c\int_{\mathbb{R}^n \backslash B_{R}(0)} \frac{|y|^{-pb_{k_0}}}{|x-y|^{n-\alpha}|y|^{\sigma_2}}\,dy \\
\geq c\int_{R}^{\infty} t^{\alpha-\sigma_2 - p b_{k_0}} \,\frac{dt}{t} = \infty. 
\end{align*}
Hence $v(x) = \infty$, which is impossible. Similarly, if there exists a $c>0$ such that 
$$v(x) \geq  c(1+|x|)^{-\theta_2} \,\text{ where }\, \theta_2 < p_0,$$ 
then we can apply the same iteration argument to conclude $u(x) = \infty$ for large $x$ and this completes the proof.
\end{proof}

\begin{proposition}\label{nonint 2} There hold the following.
\begin{enumerate}[(i)]
\item Let $\theta_3 > q_0$ and $\theta_4 > p_0$. If $u,v$ are not integrable solutions, then there does not exist a positive constant $C$ such that either
$$ u(x) \leq C(1+|x|)^{-\theta_3} \,\text{ or }\, v(x) \leq C(1+|x|)^{-\theta_4}. $$

\item If $u,v$ are not integrable solutions but are decaying solutions, i.e.
$$ u(x) \simeq |x|^{-\theta_1} \,\text{ and }\, v(x) \simeq |x|^{-\theta_2}$$
for some $\theta_{1},\theta_{2} > 0$, then they necessarily have the slow rates $\theta_1 = q_0$ and $\theta_2 = p_0$. 
\end{enumerate}
\end{proposition}

\begin{proof}
(i) On the contrary, assume there exists a $C>0$ such that $u(x) \leq C(1+|x|)^{-\theta_3}$. Then $n-r_{0}\theta_3 < 0$ and we calculate that
\begin{align*}
\int_{\mathbb{R}^n} u(x)^{r_0} \,dx = {} & \int_{B_{R}(0)} u(x)^{r_0} \,dx + \int_{\mathbb{R}^n \backslash B_{R}(0)} u(x)^{r_0} \,dx \\
\leq {} & C_{1} + C_{2}\int_{R}^{\infty} t^{n-r_{0}\theta_3} \,\frac{dt}{t} < \infty,
\end{align*}
which contradicts the assumption that $u,v$ are not integrable solutions. Likewise, if there exists a $C>0$ such that $v(x) \leq C(1+|x|)^{-\theta_4}$, a similar argument shows that $v \in L^{s_0}(\mathbb{R}^n)$, which is a contradiction.

(ii) Now suppose that $u,v$ are non-integrable solutions but are decaying solutions. Then part (i) and Proposition \ref{nonint 1} clearly imply that $u,v$ decay with the slow rates as $|x|\longrightarrow \infty$. 
\end{proof}

\begin{proof}[Proof of Theorem \ref{slow theorem}]
Part (i) of the theorem follows immediately from Proposition \ref{nonint 1} and parts (ii) and (iii) follow from Proposition \ref{nonint 2}.
\end{proof}
\small
\noindent{\bf Acknowledgment:} 
This work was completed during a visiting appointment at the University of Oklahoma, and the author would like to express his sincere appreciation to the university and the Department of Mathematics, especially Professors R. Landes and M. Zhu, for their hospitality. The author would also like to thank the anonymous referee for pointing out typographical errors and providing valuable suggestions on improving this article.


\end{document}